[1994/12/01]
\documentclass{amsart}
\chardef\bslash=`\\ 

\usepackage{graphicx}

\newcommand{\eval}[2][\right]{\relax
  \ifx#1\right\relax \left.\fi#2#1\rvert}

\bibliography{ijmsample}
\theoremstyle{plain}

\newtheorem*{conj*}{Conjecture}
\newtheorem*{cor*}{Corollary}
\newtheorem{theorem}{Theorem}[section]
\newtheorem{thm}[theorem]{Theorem}

\newtheorem{proposition}[theorem]{Proposition}
\newtheorem{corollary}[theorem]{Corollary}

\newtheorem{lema}[theorem]{Lemma}

\newtheorem{definition}{Definition}

\newtheorem{question}{Question}
\newtheorem*{pro}{Problem}

\theoremstyle{definition}
\newtheorem*{def*}{Definition}
\newtheorem{remark}[theorem]{Remark}

\newcommand{\C}{\mathcal C}

\renewcommand{\epsilon}{\varepsilon}

\newcommand{\Z}{\mathbb{Z}}
\newcommand{\N}{\mathbb{N}}
\newcommand{\R}{\mathbb{R}}
\newcommand{\eps}{\varepsilon}

\newcommand{\diam}{\operatorname{diam}}

\begin{document}

\title[Continuum-wise hyperbolicity]{Continuum-wise hyperbolicity}

\author[Artigue]{Alfonso Artigue}
\address{Departamento de Matem\'atica y Estadística del Litoral\\
Universidad de la República\\
Gral. Rivera 1350, Salto, Uruguay}
\email{artigue@unorte.edu.uy}

\author[Carvalho]{Bernardo Carvalho}
\address{Departamento de Matem\'atica\\
Universidade Federal de Minas Gerais - UFMG\\
Av. Ant\^onio Carlos, 6627 - Campus Pampulha\\
Belo Horizonte - MG, Brazil\\
Friedrich-Schiller-Universität Jena\\
Fakultät für Mathematik und Informatik\\
Ernst-Abbe-Platz 2\\
07743 Jena, Germany}
\email{bmcarvalho@mat.ufmg.br}

\author[Cordeiro]{Welington Cordeiro}
\address{Faculty of Mathematics and Computer Science\\Nicolaus Copernicus University\\
Toru\'n - Poland}
\email{wcordeiro@impan.pl}

\author[Vieitez]{Jos\'e Vieitez}
\address{Departamento de Matem\'atica y Estadística del Litoral\\
Universidad de la República\\
Gral. Rivera 1350, Salto, Uruguay}
\email{jvieitez@unorte.edu.uy}

\date{\today}

\begin{abstract}
We introduce \emph{continuum-wise hyperbolicity}, a generalization of hyperbolicity with respect to the continuum theory. We discuss similarities and differences between topological hyperbolicity and continuum-wise hyperbolicity. A shadowing lemma for cw-hyperbolic homeomorphisms is proved in the form of the L-shadowing property and a Spectral Decomposition is obtained in this scenario. In the proof we generalize the construction of Fathi \cite{Fat89} of a hyperbolic metric using only cw-expansivity, obtaining a hyperbolic cw-metric.
We also introduce cwN-hyperbolicity, exhibit examples of these systems for arbitrarily large $N\in\N$ and obtain further dynamical properties of these systems such as finiteness of periodic points with the same period.
We prove that homeomorphisms of $\mathbb{S}^2$ that are induced by topologically hyperbolic homeomorphisms of $\mathbb{T}^2$ are continuum-wise-hyperbolic and topologically conjugate to linear cw-Anosov diffeomorphisms of $\mathbb{S}^2$, being in particular cw2-hyperbolic. 
\end{abstract}
\maketitle
\tableofcontents

\section{Introduction}

In the study of chaotic dynamics a very important class of systems is formed by the \emph{Anosov diffeomorphisms} where each tangent space is divided as a direct sum of two invariant subspaces, the first being contracted and the second expanded in a uniform way by the derivative map. The simplest example of these systems can be constructed as follows: consider a hyperbolic $2\times2$ matrix $A$ with integer coefficients and determinant one and the diffeomorphism $f_A$ of the torus $\mathbb{T}^2$ induced by $A$. The diffeomorphism $f_A$ is usually called a linear Anosov diffeomorphism since it is induced by a linear map. One of the most important features of these systems is given by the Stable Manifold Theorem: local stable and local unstable sets $W^s_c(x)$ and $W^u_c(x)$ are manifolds with the same dimension as, and also tangent to, the stable and unstable spaces, respectively. These sets are elements of a pair of transverse foliations called the \emph{stable} and the \emph{unstable foliations} that together form a \emph{local product structure}: local stable and unstable sets of points $x$, $y$ sufficiently close intersect in a single point $z=W^s_c(x)\cap W^u_c(y)$. 

This was used by Bowen to prove expansiveness and the shadowing property, and from these many important properties of the hyperbolic dynamics can be obtained (see for example the monograph \cite{AH}). For this reason homeomorphisms satisfying expansiveness and shadowing are usually called \emph{topologically hyperbolic}. In previous works we have been discussing possible generalizations of this topological notion of hyperbolicity  and a central notion in these works is \emph{continuum-wise-expansiveness} defined by Kato in \cite{Kato93}.

\begin{definition}
A homeomorphism $f$ of a compact metric space $X$ is \emph{continuum-wise expansive} if there exists $c>0$ such that $W^u_c(x)\cap W^s_c(x)$ is totally disconnected for every $x\in X$.
\end{definition}

In the whole paper \emph{cw} stands for \emph{continuum-wise}. In \cite{Kato93} examples of cw-expansive homeomorphisms were given and their dynamics discussed. For example, it is proved that cw-expansive homeomorphisms defined in compact metric spaces with positive topological dimension have positive topological entropy. One important feature of cw-expansive homeomorphisms is regarding the existence of local stable and unstable continua. Assuming that $X$ is a Peano continuum, that is $X$ is a compact, connected and locally connected metric space, and denoting by $C^s_{c}(x)$ the connected component of $W^s_{c}(x)$ that contains $x$ and by $C^u_{c}(x)$ the connected component of $W^u_{c}(x)$ that contains $x$, it is proved in \cite{Kato93B} the existence of $\delta>0$ such that both $C^s_{c}(x)$ and $C^u_{c}(x)$ have diameter at least $\delta$ for every $x\in X$. We call the set $C^s_{c}(x)$ the \emph{$\eps$-stable continuum} of $x$ and $C^u_{c}(x)$ the \emph{$\eps$-unstable continuum} of $x$. An important property  (proved in Proposition 2.1 of \cite{Kato93}) of local stable and unstable continua is that they are stable and unstable, respectively. This means that $C^s_{c}(x)\subset W^s(x)$ and $C^u_{c}(x)\subset W^u(x)$ for every $x\in X$. This indicates that local stable and unstable continua of cw-expansive homeomorphisms contain some sort of hyperbolicity. We will prove this is indeed the case considering a hyperbolic cw-metric (see Section 2.2).

The same is not true for the local stable and local unstable sets as the next examples shows. The sphere $\mathbb{S}^2$ can be seen as the quotient of $\mathbb{T}^2$ by the antipodal map and then $f_A$ induces a diffeomorphism $g_A\colon \mathbb{S}^2\to\mathbb{S}^2$. We call the map $g_A$ constructed this way a \emph{linear cw-Anosov diffeomorphism}. It is transitive, cw-expansive and satisfies the shadowing property (as proved in \cite{AAV}).
In this example, there is a pair of stable and unstable singular foliations with four singularities as in the next figure.

\begin{figure}[ht]
\begin{center}
\includegraphics[scale=0.6]{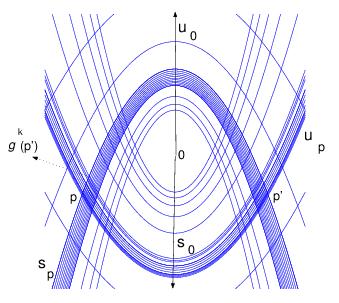}
\caption{non locally connected local stable/unstable sets}\label{fig1}
\end{center}
\end{figure}

The sets $C^s_c(x)$ and $C^u_c(x)$ are elements of this pair of singular foliations but the sets $W^s_c(x)$ and $W^u_c(x)$ are not modeled by it. Indeed, the local stable and unstable sets of most points in $S^2$ are not locally connected. To see this, choose $c>0$ and $\sigma>0$ sufficiently small such that $C^s_c(x)$ of a point $x$ that is $\sigma$-close to the singularity intersects $C^u_c(x)$ in another point $x'$ (see figure \ref{fig1}). If the point $x$ is periodic, with period $\pi(x)$, then the $g_A^{\pi(x)}$-orbit of $x'$ is contained in $W^s_c(x)$ and, hence, all the sets $C^s_c(g_A^{n\pi(x)}(x'))$, with $n\in\Z$, are contained in $W^s_{2c}(x)$. If instead of a periodic point we choose $x$ a transitive point the situation is ``worse'' in the sense that there is a Cantor set of distinct arcs contained in $W^s_c(x)$ (see Proposition 2.2.2 in \cite{ArtigueDend} for details). The same occurs with respect to $W^u_c(x)$ assuming $x$ is transitive for $g_A^{-1}$. Thus, the local stable and unstable sets are far from being a small arc, i.e.,
the sets $W^{s}_{c}(x)$ and $C^{s}_{c}(x)$, and also $W^{u}_{c}(x)$ and $C^{u}_{c}(x)$, are rather different.
The following is the main problem we have been considering.

\begin{pro}
Classify all cw-expansive homeomorphisms satisfying the shadowing property.
\end{pro}

There are intermediate phenomena that are between expansiveness and cw-expansiveness. Among the previous known notions are: countably, finite and n-expansiveness. Examples of homeomorphisms in each of these classes and satisfying the shadowing property can be found in \cite{ACCV2} and \cite{CC}. These examples all have in common the fact that the space have an infinite number of isolated points and also an infinite number of distinct chain-recurrent classes. In \cite{APV} a 2-expansive homeomorphism that is not expansive and is defined on the bitorus was introduced (this example does not satisfy the shadowing property), though it is proved there that 2-expansive homeomorphisms defined on surfaces and containing no wandering points are expansive. It is proved in \cite{ACCV2} that if the space is a closed surface and under a transitivity assumption, then countably expansive homeomorphisms with the shadowing property are topologically conjugate to a linear Anosov diffeomorphism. These results seem to indicate that when the space is regular enough, then there are now new examples of countably expansive homeomorphisms satisfying the shadowing property. Besides that, the linear cw-Anosov diffeomorphism $g_A$ is not countably expansive since it admits hyperbolic horseshoes in arbitrarily small dynamical balls. Thus, it is natural to consider generalizations of topological hyperbolicity that include this example and are different from countable expansivity. Exactly this is done with the introduction of \emph{continuum-wise hyperbolicity}.

The difference between local stable/unstable sets and local stable/unstable continua and also the ``hyperbolic'' properties of local stable/unstable continua lead us to explore a version of the local product structure that considers intersections between the local stable continuum of $x$ and the local unstable continuum of $y$ provided $x$ and $y$ are sufficiently close. This property will be called \emph{cw-local-product-structure} and generalizes the local product structure of expansive homeomorphisms to the cw-expansive scenario.

\begin{definition} \label{cwhyp}
A homeomorphism $f$ of a compact metric space $X$ satisfies the \emph{cw-local-product-structure} if for each $\eps>0$ there exists $\delta>0$ such that $$C^s_\epsilon(x)\cap C^u_\epsilon(y)\neq\emptyset \,\,\,\,\,\, \text{whenever} \,\,\,\,\,\, d(x,y)<\delta.$$ The cw-expansive homeomorphisms satisfying the cw-local-product-structure are called \emph{continuum-wise hyperbolic}
\end{definition}

Spaces supporting cw-hyperbolic homeomorphisms must be locally connected and, hence, a Peano continuum. We will usually suppress this assumption in the results, but if we say that $f$ is a cw-hyperbolic homeomorphism of $X$ it is implicit that $X$ is a Peano continuum. In this paper we study the dynamics of continuum-wise-hyperbolic homeomorphisms and exhibit examples of these systems.

\section{Topological and cw-hyperbolicity}

In this section we explore the similarities and the differences between topological hyperbolicity and cw-hyperbolicity.

\subsection{Cw-hyperbolic shadowing lemma} In hyperbolic systems there is a shadowing lemma that helps the undesrtanding of many properties of the hyperbolic dynamics.

\vspace{+0.4cm}

\hspace{-0.4cm}\emph{Hyperbolic shadowing lemma}: If $f$ is a Anosov diffeomorphism, then it satisfies the \emph{unique shadowing property}, that is,  for every $\eps>0$, there exists $\delta>0$ such that for every sequence $(x_k)_{k\in\Z}\subset X$ satisfying $$d(f(x_k),x_{k+1})\leq\delta\,\,\,\,\,\, \text{for every} \,\,\,\,\,\, k\in\Z$$ there is a unique $z\in X$ satisfying $$d(f^k(z),x_k)\leq\eps \,\,\,\,\,\, \text{for every} \,\,\,\,\,\, k\in\Z.$$ The sequence $(x_k)_{k\in\Z}$ is called a $\delta$-pseudo orbit of $f$ that is $\eps$-shadowed by $z$.

\vspace{+0.6cm}

The unique shadowing property has strong consequences on the dynamics, to name a few: local stable/unstable sets are contained in a single stable/unstable set, there is a decomposition of the non-wandering set into a finite number of chain recurrent classes containing a dense set of periodic points (see \cite{AH}), the systems satisfy other notions of pseudo-orbit tracing properties such as the limit shadowing property (see \cite{C} and \cite{CC2}). The uniqueness of the shadowing point is equivalent to expansiveness, so in the cw-hyperbolic scenario it is not possible to obtain this uniqueness. Indeed, in the non-hyperbolic examples, there are Cantor sets of distinct points shadowing the same $\delta$-pseudo orbit. But we can still prove the shadowing property and, in fact, a stronger version of it.

\vspace{+0.2cm}

\begin{definition}
A homeomorphism $f$ satisfies the L-shadowing property if for every $\eps>0$, there exists $\delta>0$ such that for every sequence $(x_k)_{k\in\Z}\subset X$ satisfying $$d(f(x_k),x_{k+1})\leq\delta\,\,\,\,\,\, \text{for every} \,\,\,\,\,\, k\in\Z \,\,\,\,\,\, \text{and}$$  $$d(f(x_k),x_{k+1})\to0 \,\,\,\,\,\, \text{when} \,\,\,\,\,\, |k|\to\infty,$$ there is $z\in X$ satisfying $$d(f^k(z),x_k)\leq\eps \,\,\,\,\,\, \text{for every} \,\,\,\,\,\, k\in\Z \,\,\,\,\,\, \text{and}$$ $$d(f^k(z),x_k)\to0\,\,\,\,\,\, \text{when} \,\,\,\,\,\, |k|\to\infty.$$ 
The sequence $(x_k)_{k\in\Z}$ is called a $\delta$-\emph{limit-pseudo-orbit} of $f$ that is $\eps$-\emph{limit-shadowed} by $z$.
\end{definition}

This property was defined in \cite{CC2} and further explored in \cite{ACCV}. Then we are ready to state our first main result:

\begin{theorem}
\label{teoChHypLsh}
Every continuum-wise-hyperbolic homeomorphism satisfies the L-shadowing property.
\end{theorem}

Using the L-shadowing property as it is done in \cite{ACCV} we can obtain information about the dynamics of every cw-hyperbolic homeomorphism proving a Spectral Decomposition in this scenario. Definitions of all notions contained in the following corollary can be found in \cite{ACCV}.

\begin{corollary}\label{classes}
If $f$ is a cw-hyperbolic homeomorphism, then
\begin{enumerate}
	\item $f$ admits only a finite number of distinct chain recurrent classes,
	\item each class is either expansive or admits arbitrarily small totally disconnected topological semihorseshoes,
	\item each chain recurrent class $C$ is transitive and admits a periodic decomposition into elementary sets $C_1,\dots,C_n$,
	\item $f^n$ restricted to each elementary set is topologically mixing and satisfies both the specification and the two-sided limit shadowing properties.
\end{enumerate}
\end{corollary}

The proof of the shadowing property for cw-hyperbolic homeomorphisms comes from the following ideas: a folklore result is that expansivity and the usual local product structure ($d(x,y)<\delta$ implies $W^u_{\eps}(x)\cap W^s_{\eps}(y)\neq\emptyset$) can be used to prove the shadowing property. The first step is to construct a \emph{hyperbolic metric} as Fathi does in \cite{Fat89} to obtain exponential rates of contractions on both $W^u_{\eps}(x)$ and $W^s_{\eps}(y)$ and the second step is to note that, with a hyperbolic metric in hands, Bowen's proof of the shadowing property in the hyperbolic scenario can be repeated using the local product structure. Our proof generalizes this idea to the cw-hyperbolic scenario: cw-expansivity together with the cw-local-product-structure can be used to prove the shadowing property. Of course we could never obtain a hyperbolic metric for cw-expansive homeomorphisms since local stable sets could not be contained on a single stable set. However, we obtain exponential rates of contraction on the local stable and local unstable continua $C^s_{\eps}(x)$ and $C^u_{\eps}(x)$ introducing a \emph{hyperbolic cw-metric}. This is not a metric on the space $X$, it is a function on the space of subcontinua with two marked points that allow us to use the cw-local-product-structure, adapt the proof of Bowen and prove the shadowing property. After the shadowing property is proved, the L-shadowing property follows from the cw-local-product-structure and Theorem E in \cite{ACCV}.

\vspace{+0.4cm}

\subsection{Hyperbolic cw-metric}

The idea of an adapted metric such that an expansive homeomorphism behaves hyperbolically with respect to that metric appears in dynamics in an article of Reddy \cite{Reddy} and was also considered by Fathi in \cite{Fat89}, although the techniques used on both articles are not the same. Fathi's idea is to build a metric taking into account the minimal number of iterates (forward or backward) that are necessary to obtain separation of a pair of points in more than an expansivity constant. If two different points $x,y$ are sufficiently close, then the number of iterates $N(x,y)$ should be large. A first idea to built such a metric would be taking $\lambda>1$ and define
$D(x,y)=\lambda^{-N(x,y)}$. In the case $x=y$ we directly put $D=0$. Of course, the problem with this naive approach is that the triangular inequality fails. Then $\lambda$ has to be chosen judiciously in order to apply Frink's metrization lemma \cite{Fri37}.
Now we adapt the idea of Fathi to the cw-expansive setting. 
Denote by $\C(X)$ the space of all subcontinua of $X$.
Define
$$E=\{(p,q,C): C\in \C(X),\, p,q\in C\}.$$
For $p,q\in C$ denote $C_{(p,q)}=(p,q,C)$.
The notation $C_{(p,q)}$ implies that $p,q\in C$ and that $C\in\C(X)$.
For a homeomorphism $f\colon X\to X$ and $C_{(p,q)}\in E$ define
$$f(C_{(p,q)})=f(C)_{(f(p),f(q))}.$$
Consider the sets
\[
\C^s_\eps(X)=\{C\in\C(X):\diam(f^n(C))\leq\eps\, \text{ for every }\, n\geq 0\},
\]
\[
\C^u_\eps(X)=\{C\in\C(X):\diam(f^{-n}(C))\leq\eps\, \text{ for every }\, n\geq 0\}.
\]
Now we are ready to state and prove our hyperbolic cw-metric.

\begin{theorem}
\label{teoCwHyp}
If $f\colon X\to X$ is a cw-expansive homeomorphism of a compact metric space $X$, then there is a function $D\colon E\to\R$ satisfying the following conditions.
\begin{enumerate}
\item Metric properties:
\vspace{+0.2cm}
\begin{enumerate}
 \item $D(C_{(p,q)})\geq 0$ with equality if, and only if, $C$ is a singleton,\vspace{+0.1cm}
 \item $D(C_{(p,q)})=D(C_{(q,p)})$,\vspace{+0.1cm}
 \item $D([A\cup B]_{(a,c)})\leq D(A_{(a,b)})+D(B_{(b,c)})$, $a\in A, b\in A\cap B, c\in B$.
\end{enumerate}
\vspace{+0.2cm}
\item Hyperbolicity: there exist constants $\lambda\in(0,1)$ and $\varepsilon>0$ satisfying
\vspace{+0.2cm}
	\begin{enumerate}
	\item if $C\in\C^s_\eps(X)$ then $D(f^n(C_{(p,q)}))\leq 4\lambda^nD(C_{(p,q)})$ for every $n\geq 0$,\vspace{+0.1cm}
	\item if $C\in\C^u_\eps(X)$ then $D(f^{-n}(C_{(p,q)}))\leq 4\lambda^nD(C_{(p,q)})$ for every $n\geq 0$.
	\end{enumerate}
\vspace{+0.2cm}
\item Compatibility: for each $\delta>0$ there is $\gamma>0$ such that
\vspace{+0.2cm}
\begin{enumerate}
\item if $\diam(C)<\gamma$, then $D(C_{(p,q)})<\delta$\,\, for every $p,q\in C$,\vspace{+0.1cm}
\item if there exist $p,q\in C$ such that $D(C_{(p,q)})<\gamma$, then $\diam(C)<\delta$.
\end{enumerate}
\end{enumerate}
\end{theorem}

\begin{proof}
In what follows we assume that $f$ is a cw-expansive homeomorphism with $\eps>0$ a cw-expansivity constant. For each non-trivial $C\in \C(X)$ we define
$$N(C)=\min\left\{|n|;\, \diam(f^nC)>\eps, \,\, n\in\Z\right\}.$$
By Corollary 2.4 in \cite{Kato93} there exists $m>0$ such that $$\diam(C)>\frac{\eps}{2} \,\,\,\,\,\, \text{implies} \,\,\,\,\,\, N(C)\leq m.$$
Let $\lambda=2^{-1/m}$ and define $\rho:\C(X)\to\R^+$ by 
$$\rho(C)=
\begin{cases}
\lambda^{N(C)} & \text{if $C$ is a non-trivial continuum}\\
0 & \text{if}\,\, C=\{p\}\,\, \text{and}\,\, p\in X.
\end{cases}$$
Define the map $D\colon E\to\R$ as follows: for each $C_{(p,q)}\in E$ let
\begin{equation}\label{Frinkmetric}
D(C_{(p,q)})=\inf \sum_{i=1}^n\rho(A^i_{(a_{i-1},a_i)})
\end{equation}
where the infimum is taken over all $n\geq 1$, $a_0=p,a_1,\dots,a_n=q\in X$ and $A^1,\dots,A^n\in \C(X)$
such that $C=\cup_{i=1}^n A^i$.
The proof of this theorem is based on the following inequalities
\begin{equation}
\label{ecu4D}
D(C_{(p,q)})\leq\rho(C)\leq 4D(C_{(p,q)})
\end{equation}
for every $C_{(p,q)}\in E$. The first inequality is trivial and the second one
will be proved in Lemma \ref{lemFrink} below.
We complete the proof assuming the inequalities in \eqref{ecu4D}.
\vspace{+0.1cm}

\emph{Metric properties}.
If $D(C_{(p,q)})=0$, then using \eqref{ecu4D} we obtain $\rho(C)=0$ and since $f$ is cw-expansive, this implies that $C$ is a singleton. This proves (a) and conditions (b) and (c) are direct from the definitions.
\vspace{+0.2cm}

\emph{Hyperbolicity}.
If $C$ is trivial then it obviously satisfies the conclusions, so we assume that $C\in\C_\eps^s(X)$ is non-trivial. Then 
$$\diam(f^n(C))\leq\eps \,\,\,\,\,\, \text{for every} \,\,\,\,\,\, n\geq 0$$
and since $\eps$ is a cw-expansivity constant of $f$ we obtain
$$\diam(f^{-N(C)}(C))>\eps.$$
Besides that, $N(f^n(C))=n+N(C)$ for every $n\geq 0$ and this assures that
\[
\begin{array}{ll}
 D(f^n(C_{(p,q)}))& \leq \rho(f^n(C))=\lambda^{N(f^n(C))}=\lambda^{n+N(C)}\\
 &=\lambda^n\lambda^{N(C)}=\lambda^n\rho(C)\\
 &\leq 4\lambda^n D(C_{(p,q)})
\end{array}
\]
for every $n\geq 0$ and every $p,q\in C$.
The hyperbolicity for $C\in\C^u_\eps(X)$ is proved in a similar way.
\vspace{+0.2cm}

\emph{Compatibility}. To prove compatibility between $\diam$ and $D$ it is sufficient to prove the compatibility betweem $\diam$ and $\rho$ in view of the inequalities \eqref{ecu4D}. For each $\delta>0$ choose $n\in\N$ such that $\lambda^n<\delta$ and choose $\gamma>0$, given by uniform continuity of $f$, such that if $C\in\C(X)$ satisfies $\diam(C)<\gamma$, then $$\diam(f^i(C))<\eps \,\,\,\,\,\, \text{whenever} \,\,\,\,\,\, |i|\leq n.$$ Thus, $N(C)> n$, $\rho(C)=\lambda^{N(C)}<\lambda^n<\delta$ and (a) is proved. For each $\delta>0$ we can choose $n\in\N$, given by Corollary 2.4 in \cite{Kato93}, such that if
$$\diam(f^i(C))<\eps \,\,\,\,\,\, \text{whenever} \,\,\,\,\,\,  |i|<n,$$ then $\diam(C)<\delta$.
If $\gamma=\lambda^n$ and $\rho(C)<\gamma$ then $\lambda^{N(C)}<\lambda^n$ and, hence, $N(C)>n$ which in turn implies $\diam(C)<\delta$ as observed above. This finishes the proof.
\end{proof}

The following result is based in \cite{Fri37} and will imply that $$\rho(C)\leq 4D(C_{(p,q)}) \,\,\,\,\,\, \text{for every} \,\,\,\,\,\, C_{(p,q)}\in E.$$

\begin{lema}
\label{lemFrink}
The function $\rho$ of the previous proof satisfies:
\begin{equation}
\label{ecuF4}
\rho(\cup_{i=1}^nC_i)\leq
  2\rho(C_1) + 4\rho(C_2)+ \dots+4\rho(C_{n-1})+2\rho(C_n)
\end{equation}
for every $C_1,\dots,C_n\in \C(M)$ such that
$C_i\cap C_{i+1}\neq\emptyset$ for every $i\in\{1,\dots,n-1\}$.
\end{lema}

\begin{proof}
First we will prove this result for $n=2$. In this case, it is enough to prove that if a non-trivial continuum $C$ can be written as a union $C=A\cup B$ where $A$ and $B$ are non-trvial continua, then
\begin{equation}
 \label{ecuGenTriang}
\rho(C)\leq 2 \max \{\rho(A),\rho(B)\}.
\end{equation}
To prove it observe that if $\diam(f^k(C))>\eps$, then either $$\diam(f^k(A))>\frac{\eps}{2} \,\,\,\,\,\, \text{or} \,\,\,\,\,\, \diam(f^k(B))>\frac{\eps}{2}$$ since $f^k(C)=f^k(A)\cup f^k(B)$.
This implies that either $$N(A)\leq m+N(C) \,\,\,\,\,\, \text{or} \,\,\,\,\,\, N(B)\leq m+N(C).$$ Indeed, in the case $\diam(f^{N(C)}(C))>\eps$ (the case $\diam(f^{-N(C)}(C))>\eps$ is analogous) the choice of $m$ assures that $$N(A)>m+N(C) \,\,\,\,\,\, \text{implies} \,\,\,\,\,\, \diam(f^{N(C)}(A))<\frac{\eps}{2}$$ and as observed this implies that $\diam(f^{N(C)}(B))>\frac{\eps}{2}$ which in turn implies that $N(B)\leq m+N(C)$. 
Assume for example that $N(A)\leq m+N(C)$. Since $0<\lambda<1$ it follows that
$$\rho(A)=\lambda^{N(A)}\geq \lambda^{m+N(C)}=\lambda^m\, \lambda^{N(C)}=\frac{1}{2}\rho(C),$$
that is, $2\rho(A)\geq \rho(C)$.
Similarly if $N(B)\leq m+N(C)$ we obtain $2\rho(B)\geq \rho(C)$ and this implies \eqref{ecuGenTriang}.
Arguing by induction, suppose that given $n\geq 3$, the conclusion of the lemma holds for every $k<n$ and let $C=\cup_{i=1}^n C_i$ with $C_i\cap C_{i+1}\neq\emptyset$ for every $i\in\{1,\dots,n-1\}$. 
Consider the following inequalities
\begin{equation}
\label{ecuCadenita}
\rho(C_1)\leq\rho(C_1\cup C_2)\leq\dots\leq \rho(C_1\cup\dots\cup C_{n-1})\leq \rho(C).
\end{equation}
If $\rho(C)\leq 2\rho(C_1)$, then \eqref{ecuF4} is proved and if $2\rho(C_1\cup\dots\cup C_{n-1})<\rho(C)$, then \eqref{ecuGenTriang}
implies that $\rho(C)\leq 2\rho(C_n)$, which also implies \eqref{ecuF4}.
Thus, we assume that $$2\rho(C_1)<\rho(C)\leq 2\rho(C_1\cup\dots\cup C_{n-1}).$$
This and \eqref{ecuCadenita} imply that there is
$1<r<n$ such that
$$2\rho(C_1\cup\dots\cup C_{r-1})<\rho(C)\leq 2\rho(C_1\cup\dots\cup C_{r}).$$
The first inequality and \eqref{ecuGenTriang} imply that
$$\rho(C)\leq 2\rho(C_r\cup\dots\cup C_{n}).$$
Thus,
\[
\rho(C)=\frac{\rho(C)}{2}+\frac{\rho(C)}{2}\leq \rho(C_1\cup\dots\cup C_{r}) + \rho(C_{r}\cup\dots\cup C_n).
\]
Since \eqref{ecuF4} holds for these two terms, by the induction assumption, the proof ends.
\end{proof}

\subsection{The shadowing property} Now we use the hyperbolic cw-metric constructed in the previous section to prove the shadowing property.

\begin{thm}
\label{teoCwHypSh}
Every cw-hyperbolic homeomorphism of a compact metric space satisfies the shadowing property.
\end{thm}

\begin{proof} In this proof we use the metric properties, the hyperbolicity of $D$ and the compatibility between $\diam$ and $D$, when necessary. Consider $\lambda\in(0,1)$ given by the hyperbolicity of $D$ and let $c>0$ be a cw-expansivity constant of $f$ (with respect to $\diam$). Let $\beta>0$ be given and choose $$\eps'=\min\left\{\frac{\beta}{3},\frac{c}{3}\right\}.$$
The compatibility of $D$ and $\diam$ assures the existence of $\eta\in(0,\eps')$ such that
$$D(C_{(p,q)})<\eta \,\,\,\,\,\, \text{for every} \,\,\,\,\,\, p,q\in C \,\,\,\,\,\, \text{implies} \,\,\,\,\,\, \diam(C)<\eps'.$$ Now choose $\xi\in(0,\eta)$ such that $$\left(\frac{4}{1-\lambda}\right)\xi<\eta$$ and use the compatibility between $D$ and $\diam$ to choose $\eps\in(0,\xi)$ such that
$$\diam(C)<2\eps \,\,\,\,\,\, \text{implies} \,\,\,\,\,\, D(C_{(p,q)})<\xi \,\,\,\,\,\, \text{for every} \,\,\,\,\,\, p,q\in C.$$
In particular, $$D(C^s_{\eps}(x)_{(p,q)})<\xi \,\,\,\,\,\, \text{and} \,\,\,\,\,\, D(C^u_{\eps}(x)_{(p',q')})<\xi$$
for every $p,q\in C^s_{\eps}(x)$, $p',q'\in C^u_{\eps}(x)$ and every $x\in X$.
Let $\delta\in(0,\eps)$, given by the cw-local product structure, be such that 
$$d(x,y)<\delta \,\,\,\,\,\, \text{implies} \,\,\,\,\,\, C^u_{\eps}(x)\cap C^s_{\eps}(y)\neq\emptyset.$$
The compatibility again assures the existence of $\gamma\in(0,\frac{\delta}{2})$ such that
$$D(C_{(p,q)})<\gamma \,\,\,\,\,\, \text{for every} \,\,\,\,\,\, p,q\in C \,\,\,\,\,\, \text{implies} \,\,\,\,\,\, \diam(C)<\frac{\delta}{2}.$$ Choose $n>0$ large enough to obtain $4\lambda^n\xi<\gamma$ and finally choose $\alpha\in(0,\gamma)$, given by uniform continuity of $f$, such that if $(y_i)_{i=0}^n$ is an $\alpha$-pseudo orbit of $f$ with length $n$, then $$d(f^i(y_0),y_i)<\frac{\delta}{2} \,\,\,\,\,\, \text{for every} \,\,\,\,\,\, i\in\{0,\dots,n\}.$$ To prove that every infinite $\alpha$-pseudo orbit $(x_i)_{i\in\Z}$ is $\beta$-shadowed it is enough to prove this for $(x_i)_{i=0}^{rn}$ for any $r\in\N$. As a first step, let $x'_{rn}=x_{rn}$ and note that $$d(f^n(x_{(r-1)n}),x'_{rn})<\delta.$$ The cw-local product structure assures the existence of $$y'_{rn}\in C^u_{\eps}(f^n(x_{(r-1)n}))\cap C^s_{\eps}(x'_{rn})$$ and then we can define $x'_{(r-1)n}=f^{-n}(y'_{rn})$. To simplify the notation, let $$C^r=C^u_{\eps}(f^n(x_{(r-1)n})).$$ It follows from the hyperbolicity of $D$ that
$$D(f^{-n}(C^r_{(p,q)}))\leq 4\lambda^nD(C^r_{(p,q)})\leq 4\lambda^n\xi <\gamma \,\,\,\,\,\, \text{for every} \,\,\,\,\,\, p,q\in C^r.$$
The compatibility of $\diam$ and $D$ assures that $$\diam(f^{-n}(C^r))<\frac{\delta}{2},$$ and, hence,
\begin{eqnarray*}
d(x'_{(r-1)n},x_{(r-1)n})&=&d(f^{-n}(y'_{rn}),f^{-n}(f^n(x_{(r-1)n})))\\
&\leq& \diam(f^{-n}(C^r))<\frac{\delta}{2}.
\end{eqnarray*}
Since $$d(f^n(x_{(r-2)n}),x_{(r-1)n})<\frac{\delta}{2},$$ it follows that $$d(f^n(x_{(r-2)n}),x'_{(r-1)n})<\delta.$$ Choose $$y'_{(r-1)n}\in C^u_{\eps}(f^n(x_{(r-2)n}))\cap C^s_{\eps}(x'_{(r-1)n})$$ and let $x'_{(r-2)n}=f^{-n}(y'_{(r-1)n})$. By induction, we construct the sequence $(x'_{in})_{i=0}^r$.
We will prove that $x=x'_0$ is the shadowing point of $(x_i)_{i=0}^{rn}$. For $0\leq i\leq rn$ let $s\geq0$ be such that $sn\leq i<(s+1)n$. Note that $d(f^i(x),x_i)$ is bounded above by $$d(f^i(x),f^{i-sn}(x'_{sn}))+d(f^{i-sn}(x'_{sn}),f^{i-sn}(x_{sn}))+d(f^{i-sn}(x_{sn}),x_i).$$ The second term is smaller than $\frac{\beta}{3}$ since $f^n(x'_{sn})\in C^u_{\eps}(f^n(x_{sn}))$ and $i-sn<n$. The third term is also smaller than $\frac{\beta}{3}$ by the choice of $\alpha$. To see that the first term is bounded by $\frac{\beta}{3}$, consider the set $$\mathcal{C}=\bigcup_{j=0}^sf^{i-jn}(C_{\eps}^s(x'_{jn})).$$
Note that $\mathcal{C}$ is a continuum since for each $j\in\{0,\dots,s\}$ we have 
$$f^{i-jn}(x_{jn}')\in f^{i-jn}(C_{\eps}^s(x'_{jn}))\cap f^{i-(j-1)n}(C_{\eps}^s(x'_{(j-1)n})).$$ 
The second inclusion holds because the cw-local-product-structure assures that $$x'_{jn}\in C_{\eps}^s(f^n(x'_{(j-1)n})).$$ The hyperbolicity of $D$ assures that $$D(f^{i-jn}(C^s(x'_{jn})_{(p,q)}))\leq 4\lambda^{i-jn}\xi$$ 
for every $p,q\in C^s(x'_{jn})$ and every $j\in\{0,\dots,s\}$. 
Thus, we can use the metric and hyperbolic properties of $D$ to obtain
\begin{eqnarray*}
D(\mathcal{C}_{(f^i(x),f^{i-sn}(x'_{sn}))})&\leq&D(f^{i}(C^s_{\eps}(x))_{(f^{i}(x),f^{i}(x))})\\
&+&\sum_{j=1}^{s}D(f^{i-jn}(C^s_{\eps}(x'_{jn}))_{(f^{i-(j-1)n}(x_{(j-1)n}'),f^{i-jn}(x_{jn}'))})\\
&\leq& \sum_{j=0}^s4\lambda^{i-jn}\xi=4\lambda^{i-sn}\xi\sum_{j=0}^s\lambda^{(s-j)n}\\
&\leq&\left(\frac{4}{1-\lambda^n}\right)\xi<\left(\frac{4}{1-\lambda}\right)\xi<\eta,
\end{eqnarray*}
which, in turn, implies that $\diam(\mathcal{C})<\eps'$. Finally, note that both $f^i(x)$ and $f^{i-sn}(x'_{sn})$ belong to $\mathcal{C}$ and, therefore, $$d(f^i(x),f^{i-sn}(x'_{sn}))\leq\diam(\mathcal{C})<\eps'<\frac{\beta}{3}$$ completing the proof of the shadowing property.
\end{proof}

\section{CwN-hyperbolicity} In the definition of cw-local-product-structure we ask that $\eps$-stable and $\eps$-unstable continua of sufficiently close points intersect, but nothing is said about the set of intersections between these continua. At first, they could intersect in just a single point, or a cantor set. We can distinguish cw-hyperbolic homeomorphisms with respect to the cardinality of these sets of intersections through the following definitions.

\begin{definition}
A cw-expansive homeomorphism $f$ is said to be cwN-expansive if there exists $c>0$ such that
$$\#(C^s_c(x)\cap C^u_c(x))\leq N \,\,\,\,\,\, \text{for every} \,\,\,\,\,\, x\in X.$$
We say that $f$ is \emph{cwN-hyperbolic} if it is cw-hyperbolic and cwN-expansive. In this case we say that $c$ is a cwN-expansivity constant.
\end{definition}

It is easy to note that expansive homeomorphisms are cw1-expansive. The linear cw-Anosov diffeomorphisms of $\mathbb{S}^2$ are cw2-hyperbolic and not cw1-hyperbolic. We prove in Theorem \ref{thmproduct} and Corollary \ref{corproduct} that we can construct cwN-hyperbolic homeomorphisms with arbitrarily large $N\in\N$. We can similarly define cw(finite)-expansivity and cw(countably)-expansivity and then define cw(finite)-hyperbolicity and cw(countably)-hyperbolicity, though we still do not have examples on these classes that are not cwN-hyperbolic for any $N\in\N$.

Note the differences between these definitions and the definitions of N-expansivity, finite-expansivity and countable-expansivity. The latters impose restrictions on the dynamical balls $W^s_c(x)\cap W^u_c(x)$, asking that they contain at most N points, a finite or a countable set. In the definition of cwN-expansivity, the restrictions are on the set $C^s_c(x)\cap C^u_c(x)$ which can be seen as a dynamical ball by local stable/unstable continua. Also, using the L-shadowing property as in Corollary C of \cite{ACCV} we can prove that cw-hyperbolic homeomorphisms, that are not topologically hyperbolic, cannot be countably expansive, neither entropy expansive. The following theorem is based on the arguments of \cite{ArtigueDend} and \cite{Hiraide}.

\begin{theorem}\label{cwN-hyperbolic}
A homeomorphism $f$ of a compact metric space $X$ is cw1-hyperbolic if, and only if, $f$ is topologically hyperbolic and $X$ is a Peano continuum.
\end{theorem}

\begin{proof}
Suppose first that $X$ is a Peano continuum and that $f$ is expansive, with expansivity constant $c>0$, and satisfies the shadowing property. To prove cw1-hyperbolicity it is enough to prove the cw-local product structure since expansivity clearly implies cw1-expansivity. Let $\eps_0=\frac{c}{4}$ and for each $0<\eps<\eps_0$, let $\delta\in(0,\eps)$ be given by the shadowing property such that
$$d(x,y)<\delta \,\,\,\,\,\, \text{implies} \,\,\,\,\,\, W^u_{\eps}(x)\cap W^s_{\eps}(y)\neq\emptyset.$$
For each $x\in X$, let $C^s_{\eps,\delta}(x)$ denote the connected component of $W^s_{\eps}(x)\cap B(x,\delta)$ that contains $x$ and also $C^u_{\eps,\delta}(x)$ denote the connected component of $W^u_{\eps}(x)\cap B(x,\delta)$ that contains $x$.
Since $X$ is a Peano continuum, Proposition 1.1  of \cite{Hiraide} and also the first paragraph of the proof of Lemma 1.3 in the same paper, assure the existence of $\rho\in(0,\delta)$ such that $$W^s_{\rho}(x)\subset C^s_{\eps,\delta}(x) \,\,\,\,\,\, \text{and} \,\,\,\,\,\, W^u_{\rho}(x)\subset C^u_{\eps,\delta}(x) \,\,\,\,\,\, \text{for every} \,\,\,\,\,\, x\in X.$$
Let $\delta'\in(0,\rho)$, given by the shadowing property, be such that $$d(x,y)<2\delta' \,\,\,\,\,\, \text{implies} \,\,\,\,\,\, W^u_{\rho}(x)\cap W^s_{\rho}(y)\neq\emptyset.$$
Since $f$ is expansive, the set $W^u_{\rho}(x)\cap W^s_{\rho}(y)$ is a singleton and we denote this point by $[x,y]$. For each $x\in X$, this defines a map
$$[.,.]\colon B(x,\delta')\times B(x,\delta')\to X$$ and we consider the restriction $$[.,.]\colon C^s_{\eps,\delta'}(x)\times C^u_{\eps,\delta'}(x)\to N_x,$$ where $N_x:=[C^s_{\eps,\delta'}(x),C^u_{\eps,\delta'}(x)]$. As in Proposition 1.1 of \cite{Hiraide}, this restriction is a homeomorphism and, hence, $N_x$ is connected and open in $X$ and there exists $\rho'>0$ satisfying $$B(x,\rho')\subset N_x \,\,\,\,\,\, \text{for every} \,\,\,\,\,\, x\in X.$$ 
Thus, if $d(x,y)<\rho'$, then $y\in N_x$ and there exists $w\in C^s_{\eps,\delta'}(x)$ and $z\in C^u_{\eps,\delta'}(x)$ such that $[w,z]=y$. Thus, 
$$w\in C^s_{\eps,\delta'}(x)\cap W^u_{\rho}(y) \,\,\,\,\,\, \text{and} \,\,\,\,\,\, z\in C^u_{\eps,\delta'}(x)\cap W^s_{\rho}(y).$$ Since
$$W^u_{\rho}(y)\subset C^u_{\eps,\delta}(y) \,\,\,\,\,\, \text{and} \,\,\,\,\,\, W^s_{\rho}(z)\subset C^s_{\eps,\delta}(z),$$
it follows that $$w\in C^s_{\eps,\delta'}(x)\cap  C^u_{\eps,\delta}(y)  \,\,\,\,\,\, \text{and} \,\,\,\,\,\, z\in C^u_{\eps,\delta'}(x)\cap C^s_{\eps,\delta}(z).$$
This proves, in particular, that for each $\eps>0$ there is $\rho'>0$ such that $d(x,y)<\rho'$ implies $$C^s_{\eps}(x)\cap C^u_{\eps}(y)\neq\emptyset \,\,\,\,\,\, \text{and} \,\,\,\,\,\, C^u_{\eps}(x)\cap C^s_{\eps}(y)\neq\emptyset$$ and the cw-local product structure is proved.

Now, assume that $f$ is cw1-hyperbolic and we will prove that $f$ is expansive (the shadowing property is proved for every cw-hyperbolic homeomorphism). Let $c>0$ be such that
$$C^s_c(x)\cap C^u_c(x)=\{x\} \,\,\,\,\,\, \text{for every} \,\,\,\,\,\, x\in X.$$
Let $0<\eps<\frac{c}{2}$, given by uniform continuity of $f$, be such that $\diam(A)<2\eps$ implies
$$\max\{\diam(f(A)),\diam(f^{-1}(A))\}<\frac{c}{2}.$$
Consider $0<\delta<\eps$, given by the cw-local-product-structure, such that
$$d(x,y)<\delta \,\,\,\,\,\, \text{implies} \,\,\,\,\,\, C^u_{\eps}(x)\cap C^s_{\eps}(y)\neq\emptyset$$
and also $\delta'>0$ such that
$$d(x,y)<\delta' \,\,\,\,\,\, \text{implies} \,\,\,\,\,\, C^u_{\delta}(x)\cap C^s_{\delta}(y)\neq\emptyset.$$
We will prove that $\delta'$ is an expansive constant of $f$. Indeed, if there exists $x\in X$ and $y\in\Gamma(x,\delta')\setminus\{x\}$, then we can choose
$$z\in C^u_{\delta}(x)\cap C^s_{\delta}(y).$$ Since $x\neq y$ it follows that either $z\neq y$ or $z\neq x$. We deal with the first case and note that the second is analogous. We observe that there exists $n\in\N$ such that $$f^{-n}(y)\notin C^s_{\eps}(f^{-n}(z)).$$ Indeed, if $f^{-n}(y)\in C^s_{\eps}(f^{-n}(z))$ for every $n\in\N$, then
$$d(z,y)\leq\lim_{n\to\infty}\diam f^n(C^s_{\eps}(f^{-n}(z)))$$
and this limit is 0 by Proposition 2.1 in \cite{Kato93}, which contradicts $z\neq y$. Then we can choose $n_0\in\N$ the least natural number such that  $f^{-n_0}(y)\notin C^s_{\eps}(f^{-n_0}(z))$. Since
$$d(f^{-n_0}(z),f^{-n_0}(y))\leq\delta$$ the cw-local-product-structure applies again to ensure the existence of
$$z'\in C^u_{\eps}(f^{-n_0}(z))\cap C^s_{\eps}(f^{-n_0}(y)).$$
Thus, $$z'\in C^u_c(f^{-n_0}(z))\cap C^s_c(f^{-n_0}(z)),$$ since $f^{-n_0}(y)\in C^s_{\frac{c}{2}}(f^{-n_0}(z))$, and $z'\neq f^{-n_0}(z)$ because $z'\in C^s_{\eps}(f^{-n_0}(y))$ and $f^{-n_0}(z)\notin C^s_{\eps}(f^{-n_0}(y))$. This contradicts cw1-expansivity and proves (1).
\end{proof}

Now we introduce examples of cwN-hyperbolic homeomorphisms for arbitrarily large $N\in\N$. In what follows we endow the set $X\times Y$ with the maximum metric and consider the homeomorphism $f\times g\colon X\times Y\to X\times Y$ defined by $$[f\times g](x,y)=(f(x),g(y)).$$

\begin{theorem}\label{thmproduct}
Let $X,Y$ be compact metric spaces, $f\colon X\to X$ and $g\colon Y\to Y$ be homeomorphisms. The following hold:
\begin{enumerate}
	\item If $f$ and $g$ are cw-expansive, then $f\times g$ is cw-expansive.
	\item If $f$ and $g$ are cw-hyperbolic, then $f\times g$ is cw-hyperbolic.
	\item If $f$ is cwN-hyperbolic and $g$ is cwM-hyperbolic, then $f\times g$ is cwNM-hyperbolic.
	\item If, in addition, $f$ is not cw(N-1)-expansive and $g$ is not cw(M-1)-expansive, then $f\times g$ is not cw(NM-1)-expansive.
\end{enumerate}
\end{theorem}

\begin{proof}
Let $\pi_1$ and $\pi_2$ denote the projections of $X\times Y$ on $X$ and $Y$, respectively. Since these are continuous maps, the image of any continuum in $X\times Y$ is also a continuum in $X$ or $Y$. Let $\eps>0$ be the minimun of the constants given by cw-expansivity of $f$ and $g$ and consider $\eps'>0$ the minimum of the constants given by the uniform continuity of $\pi_1$ and $\pi_2$ for $\eps$. We will prove that $\eps'$ is a cw-expansivity constant for $f\times g$. Indeed, if $C\subset X\times Y$ is a non-trivial continuum satisfying
$$\diam([f\times g]^k(C))\leq\eps' \,\,\,\,\,\, \text{for every} \,\,\,\,\,\, k\in\Z,$$ then $\pi_1(C)\subset X$ is a continuum satisfying
$$\diam(f^k(\pi_1(C)))=\diam(\pi_1([f\times g]^k(C)))\leq\eps \,\,\,\,\,\, \text{for every} \,\,\,\,\,\, k\in\Z$$
and $\pi_2(C)\subset Y$ is a continuum satisfying
$$\diam(g^k(\pi_2(C)))=\diam(\pi_2([f\times g]^k(C)))\leq\eps \,\,\,\,\,\, \text{for every} \,\,\,\,\,\, k\in\Z.$$ Since $C$ is non-trivial, either $\pi_1(C)$ or $\pi_2(C)$ are non-trivial and this contradicts cw-expansivity of either $f$ or $g$, so (1) is proved. Now, for each $\eps>0$, let $\delta>0$ be the minimum of the constants given by the cw-local-product-structure of $f$ and $g$. We will prove that $\delta$ is the constant of the cw-local-product-structure for $f\times g$. If $x$ and $y$ are $\delta$-close in $X\times Y$, then $\pi_1(x),\pi_1(y)$ are $\delta$-close in $X$ and $\pi_2(x),\pi_2(y)$ are $\delta$-close in $Y$. The cw-local-product-structure of $f$ and $g$ assure the existence of
$$z_1\in C^u_{\eps}(\pi_1(x))\cap C^s_{\eps}(\pi_1(y)) \,\,\,\,\,\, \text{and} \,\,\,\,\,\, z_2\in C^u_{\eps}(\pi_2(x))\cap C^s_{\eps}(\pi_2(y)).$$
Thus,
$$(z_1,z_2)\in [C^u_{\eps}(\pi_1(x))\times C^u_{\eps}(\pi_2(x))]\cap [C^s_{\eps}(\pi_1(y))\times C^s_{\eps}(\pi_2(y))]$$
and, hence, $(z_1,z_2)$ belongs to the intersection of the $\eps$-unstable continuum of $x$ and the $\eps$-stable continuum of $y$ for $f\times g$. This proves (2). To prove (3), let $\eps>0$ be the minimun of the constants given by cwN-expansivity of $f$ and cwM-expansivity of $g$ and consider $\eps'>0$ the minimum of the constants given by the uniform continuity of $\pi_1$ and $\pi_2$ for $\eps$. We will prove that $\eps'$ is a cwNM-expansivity constant for $f\times g$. Let $z_1,\dots,z_{NM+1}$ be distinct points that belong to $C^u_{\eps'}(z_1)\cap C^s_{\eps'}(z_1)$. Then as in the proof of (1) it follows that $$\pi_1(z_i)\in C^u_{\eps}(\pi_1(z_1))\cap C^s_{\eps}(\pi_1(z_1)) \,\,\,\,\,\, \text{and} \,\,\,\,\,\, \pi_2(z_i)\in C^u_{\eps}(\pi_2(z_1))\cap C^s_{\eps}(\pi_2(z_1))$$
for every $i\in\{1,\dots,NM+1\}$. Since $f$ is cwN-expansive, there exist $z_{i_1},\dots,z_{i_N}$ such that
$$\pi_1(z_i)=\pi_1(z_{i_j}) \,\,\,\,\,\, \text{for some} \,\,\,\,\,\, j\in\{1,\dots,N\}.$$
Note that  if $k\neq i_j$ and $\pi_1(z_k)=\pi_1(z_{i_j})$, then $\pi_2(z_k)\neq\pi_2(z_{i_j})$, since $z_k\neq z_{i_j}$.
Also, there exists at least one $j\in\{1,\dots,N\}$ such that $\pi_1^{-1}(z_{i_j})$ contains M+1 distinct points of $z_1,\dots,z_{NM+1}$.
These contradict the cwM-expansivity of $g$ and proves (3). Now, suppose that $f$ is not cw(N-1)-expansive and $g$ is not cw(M-1)-expansive. Then for every $\eps>0$ there exist $x_1,\dots,x_N\in X$ and $y_1,\dots,y_M\in Y$ such that
$$x_i\in C^u_{\eps}(x_1)\cap C^s_{\eps}(x_1) \,\,\,\,\,\, \text{and} \,\,\,\,\,\, y_j\in C^u_{\eps}(y_1)\cap C^s_{\eps}(y_1)$$ and as in the proof of (2),
$$(x_i,y_j)\in C^u_{\eps}(x_1,y_1)\cap C^s_{\eps}(x_1,y_1)$$
for every $i\in\{1,\dots,N\}$ and $j\in\{1,\dots,M\}$. Since this can be done for each $\eps>0$, it follows that $f\times g$ is not cw(NM-1)-expansive and finishes the proof.
\end{proof}

\begin{corollary}\label{corproduct}
For each $N\in\N$ there exists a cw$(2^N)$-hyperbolic homeomorphism that is not cw$(2^N-1)$-expansive.
\end{corollary}

\begin{proof}
Consider the product of $N$ copies of a linear cw-Anosov homeomorphism of $\mathbb{S}^2$. Since each copy is cw2-hyperbolic, the product is cw$(2^N)$-hyperbolic by item (3) of Theorem \ref{thmproduct} and since each copy is not cw1-expansive, the product is not cw$(2^N-1)$-expansive by item (4) of the same theorem.\end{proof}

The following diagram of implications enlighten the differences observed above and summarize the results obtained so far.

\begin{eqnarray*}  expansive + shadowing  \ \ \ \  \ \ \ \ \ &\Leftrightarrow& \ \ \ \ \ \ \  cw1-hyperbolic  \\
 \Downarrow \ \ \ \ \ \ \ \ \ \ \ \ \ \ \ \ \ \ \ \ \ \ \ \ \ \ \ \ \ \ \ &  &  \ \ \ \ \ \ \ \ \ \ \ \ \ \ \ \ \ \ \    \Downarrow   \\
 N-expansive + shadowing  \ \ \ \ \ \ \ \ \ & &   \ \ \ \ \ \ \ \ cwN-hyperbolic \\
\Downarrow \ \ \ \ \ \ \ \ \ \ \ \ \ \ \ \ \ \ \  \ \ \ \ \ \ \ \ \ \ \ \ & & \ \ \ \ \ \ \ \ \ \ \ \ \ \ \ \ \ \ \ \   \Downarrow \\
 N+1-expansive + shadowing \ \ \ \ \ \ \ & & \ \ \ \   cw(N+1)-hyperbolic \ \ \ \ \ \ \ \ \ \ \ \ \ \ \ \  \\  \Downarrow \ \ \ \ \ \ \ \ \ \ \ \ \ \ \ \ \ \ \ \ \ \ \ \ \ \ \ \ \ \ \ &  & \ \ \ \ \ \ \ \ \ \ \ \ \ \ \ \ \ \ \ \   \Downarrow   \\
 finite \ expansive + shadowing \ \ \ \ \ \ \ & & \ \ \ \  cw(finite)-hyperbolic \ \ \ \ \ \ \ \ \ \ \ \ \ \ \ \  \\  \Downarrow \ \ \ \ \ \ \ \ \ \ \ \ \ \ \ \ \ \ \ \ \ \ \ \ \ \ \ \ \ \ \ &  & \ \ \ \ \ \ \ \ \ \ \ \ \ \ \ \ \ \ \ \   \Downarrow   \\
countably \ expansive + shadowing \ \ \ \ \ \ \ & &    cw(countably)-hyperbolic \ \ \ \ \ \ \ \ \ \ \ \ \ \ \ \  \\
 \Downarrow \ \ \ \ \ \ \ \ \ \ \ \ \ \ \ \ \ \ \ \ \ \ \ \ \ \ \ \Downarrow  \ \ \ \ \ \ \ \ \ \ \ \ \ \  &  & \ \ \ \ \ \ \ \ \ \ \ \ \ \ \ \ \ \ \ \ \  \Downarrow   \\
entropy \ expansive \ \ \ \ \ \ \ \ \ \ \ cw-expansive \ \ &\Leftarrow& \ \ \ \ \ \ \ \ \ \     cw-hyperbolic \ \ \ \ \ \ \ \ \ \ \ \ \   \ \\
+ shadowing  \ \ \ \ \ \ \ \ \ \ \ + shadowing  \ \ \ \  & & \ \ \ \ \ \ \ \ \ \      \ \ \ \ \ \ \ \ \ \ \ \ \
\\
& &
\end{eqnarray*}

Let us briefly explain the implications on this diagram. On the left side we put homeomorphisms satisfying generalizations of expansiveness and the shadowing property and the chain of implications follow directly from the definitions: expansivity implies N-expansivity, that implies (N+1)-expansivity, which implies finite-expansivity, countable expansivity and cw-expansivity. For each $N\in\N$, examples of N-expansive homeomorphisms that are not (N-1)-expansive and satisfying the shadowing property can be found in \cite{CC}. In the same paper there are examples of finite expansive homeomorphisms with the shadowing property that are not N-expansive for any $N\in\N$. In \cite{ACCV2} there are examples of countably expansive homeomorphisms satisfying the shadowing property that are not finite expansive. All these examples are entropy and cw-expansive homeomorphisms. In \cite{ACCV2} we can also find examples of entropy expansive homeomorphisms with the shadowing property that are not countably expansive and in \cite{AAV} and \cite{ArtigueDend} it is proved that the linear cw-Anosov diffeomorphisms are cw-expansive, satisfy the shadowing property but are not countably nor entropy-expansive. On the right side we put the cw-hyperbolic homeomorphisms and the following chain of implications follow directly from the definitions: cw1-hyperbolicity implies cwN-hperbolicity, that implies cw(N+1)-hyperbolicity, which implies cw(finite)-hyperbolicity, cw(countable)-hyperbolicity and cw-hyperbolicity. The linear cw-Anosov diffeomorphisms are cw2-hyperbolic but not cw1-hyperbolic and for arbitrarily large $N\in\N$ examples of cwN-hyperbolic homeomorphisms that are not cw(N-1)-hyperbolic are given in Corollary \ref{corproduct}. The column on the right side is related to the column on the left side in the beginning, where expansiveness and shadowing are equivalent to cw1-hyperbolicity as proved in Theorem \ref{cwN-hyperbolic}, and in the end where cw-hyperbolicity implies both cw-expansiveness and the shadowing property as proved in Theorem \ref{teoCwHypSh}. The middle parts of both columns are different since cw-hyperbolic homeomorphisms are either expansive or they contain arbitrarily small topological semihorseshoes, as proved in Corollary \ref{classes}, and, hence, they are not countably nor entropy-expansive. This shows that the approach we follow in this paper is completely different from what we were doing in the previous ones (\cite{ACCV}, \cite{ACCV2}, \cite{APV}, \cite{CC} and \cite{CC2}) though it is still related to the idea of generalizing topological hyperbolicity and characterizing cw-expansive homeomorphisms satisfying the shadowing property.

\section{Cw-hyperbolicity and periodic points}
Now we discuss about the set of periodic points of cw-hyperbolic homeomorphisms. We recall that the topologically hyperbolic shadowing lemma is in the form of the unique shadowing property, while the cw-hyperbolic shadowing lemma is in the form of the L-shadowing property. One important difference between these two properties is regarding the existence of periodic points. In the unique shadowing property one can easily prove that the periodic points are dense in the non-wandering set observing that if $x$ and $f^k(x)$ are sufficiently close, then there is an unique point shadowing the segment of orbit from $x$ to $f^k(x)$ indefinitely and that this point must be periodic. In \cite{ACCV} there are examples of homeomorphisms satisfying the L-shadowing property but without periodic points, so the following question seems natural:

\begin{question}
Does cw-hyperbolicity imply the density of periodic points in the non-wandering set? Or simply, does it imply the existence of periodic points?
\end{question}

In the case where the space is a closed surface then it is proved in \cite{KP} that the shadowing property is sufficient for the existence of periodic points in each $\eps$-transitive component, even though nothing is said about the density of periodic points in the non-wandering set. As a consequence we obtain the following:

\begin{theorem}
Cw-hyperbolic homeomorphisms of surfaces contain at least one periodic point in each of its chain recurrent classes.
\end{theorem}

\begin{proof}
Since cw-hyperbolic homeomorphisms have a finite number of chain recurrent classes (by Corollary \ref{classes}) then there exists $\eps>0$ such that each chain-recurrent class is also an $\eps$-transitive class. Then Theorem 1.1 in \cite{KP} assures that each one of the classes contain at least one periodic point.\end{proof}

We do not have information on the density of periodic points in the non-wandering set even in this case of a surface. The phenomena in Corollary \ref{classes} of the existence of arbitrarily small topological semihorseshoes could be very different from the hyperbolic scenario. These could be constructued to be approximating the chain-recurrent classes but could not contain periodic points. A different but related question is regarding the existence of an infinite number of periodic points with the same period. Another consequence of expansiveness is that for each $n\in\N$ there is only a finite number of periodic points with period $n$. Thus, the following question is natural:

\begin{question}
Does there exist a cw-hyperbolic homeomorphism with an infinite number of periodic points with the same period? And with an infinite number of fixed points? 
\end{question}

We answer this question negatively assuming cwN-hyperbolicity.

\begin{theorem}
If a homeomorphism is cwN-hyperbolic, then it has only finitely many periodic points with the same period. In particular, it has only finitely many fixed points.
\end{theorem}

\begin{proof}
Let $\gamma>0$ be given by cwN-expansivity such that 
$$\#(C^s_{2\gamma}(x)\cap C^u_{2\gamma}(x))\leq N \,\,\,\,\,\, \text{for every} \,\,\,\,\,\, x\in X.$$
For each $k\in\N$, the uniform continuity of $f$ assures the existence of $\delta>0$ such that if $\diam(A)<2\delta$ then $$\diam(f^i(A))<\gamma \,\,\,\,\,\, \text{for every} \,\,\,\,\,\, -Nk\leq i\leq Nk.$$ Now choose $0<\eps<\delta$, given by the cw-local-product-structure, such that $$d(x,y)<\eps \,\,\,\,\,\, \text{implies} \,\,\,\,\,\, C^u_\delta(x)\cap C^s_\delta(y)\neq\emptyset.$$ We will prove that if $p$ is a periodic point of $f$ with period $k$, then there is no other periodic point of $f$ with period $k$ in $B(p,\eps)$. Indeed, consider $q\in B(p,\eps)$ a periodic point of $f$ with period $k$ and let $z\in C_\delta^u(p)\cap C_\delta^s(q)$. Note that the choice of $\delta$ and $f^{-k}(q)=q$ assure that
$$f^{-ik}(C^s_\delta(q))\subset C^s_\gamma(q) \,\,\,\,\,\, \text{for every} \,\,\,\,\,\, i\in\{0,\dots,N\}.$$
In particular, $$f^{-ik}(z)\in C_\gamma^u(p)\cap C_\gamma^s(q) \,\,\,\,\,\, \text{for every} \,\,\,\,\,\,  i\in\{0,\dots,N\}.$$
Hence, $$\#(C^s_{2\gamma}(z)\cap C^u_{2\gamma}(z))\geq N+1$$ contradicting cwN-expansivity.
\end{proof}

\begin{question}
Does there exists a cw-hyperbolic homeomorphism that is not cwN-hyperbolic for any $N\in\N$?
\end{question}

There is an example in \cite{AAV} of a cw-expansive homeomorphism of $\mathbb{T}^2$ with an infinite number of fixed points. We will prove that it is not cw-hyperbolic. The example is made by $C^0$-small perturbations, so called ``modifications'' in \cite{AAV}, made in the linear Anosov diffeomorphism $f_A$ given by the action in $\mathbb{T}^2$ of  $$A=\left(
                                                                 \begin{array}{cc}
                                                                   2 & 1 \\
                                                                   1 & 1 \\
                                                                 \end{array}
                                                               \right).$$ 
Consider the two eigenvalues of $A$: $$\lambda=\frac{3+\sqrt{5}}{2}>1 \,\,\,\,\,\, \text{and} \,\,\,\,\,\, \lambda^{-1}=\frac{3-\sqrt{5}}{2}<1$$ and its respective eigenvectors $$\left(1,\frac{\sqrt{5}-1}{2}\right) \,\,\,\,\,\, \text{and} \,\,\,\,\,\, \left(-1,\frac{\sqrt{5}+1}{2}\right)$$ which form an orthogonal basis of $\mathbb{R}^2$.
After a linear orthogonal change of coordinates $P$ we may assume that the linear map induced by $A$ is given by
$$\left(\begin{array}{l}
     2x+y \\
     x+y \end{array}\right)  =P^{-1}\circ T \circ P
                                              \left(\begin{array}{c}
                                              x \\
                                              y \\
                                              \end{array} \right)$$
where $T(x,y)= (\lambda x,\lambda^{-1} y)$.
Using Oxtoby-Ulam Theorem (Corollary 3 in \cite{OU}) the authors in \cite{AAV} defined the modifications $T_n$ in a domain $H_n$ contained in $$\{(x,y)\in\mathbb{R}^2;\,\, x\geq 0, \,\, y\geq 0\}$$ and limited by the branch of hyperboles 
$$xy=\frac{1}{2^{2n-1}}\,\,\,\,\,\, \text{and} \,\,\,\,\,\, xy=\frac{1}{2^{2n+1}}$$ and the straight lines $y=\lambda^3 x$ and $y=\lambda^{-3} x$ (see Figure 1 in \cite{AAV}). 
The modifications are done in such a way that there is a single fixed point $u_n=(\frac{1}{2^n},\frac{1}{2^n})$ in $xy=\frac{1}{2^{2n}}\cap H_n$ and such that $T_n=T$ in the boundary of $H_n$. 
After performing all these modifications $T_n$ with $n\in\N$, we end with a homeomorphism $f:\mathbb{T}^2\to\mathbb{T}^2$ satisfying the following condition: there exists $\eps>0$ such that
\begin{equation}
 \label{eqPertCwExp}
 \text{if }\,\,\ \diam f^n(C)\leq\eps \,\, \text{ for every } \,\, n\in\Z, \,\, \text{ then } \,\,\diam C\leq\eps/2
\end{equation}
for every continuum $C\subseteq \mathbb{T}^2$ (see Section 4 in \cite{AAV}).
This property let us define the equivalence relation $\sim$ in $\mathbb{T}^2$ as follows: $p\sim q$ if there exists $C\in\C(\mathbb{T}^2)$ such that $p,q\in C$ and $$\diam f^n(C)\leq\eps \,\,\,\,\,\, \text{for every} \,\,\,\,\,\, n\in\Z.$$
We consider the quotient space $\mathbb{T}^2/\sim$ and the map $g\colon\mathbb{T}^2/\sim \, \to \mathbb{T}^2/\sim$ induced by $f$: if $[x]\in\mathbb{T}^2/\sim$ denotes the equivalence class of $x\in \mathbb{T}^2$ under $\sim$, then $g([x])=[f(x)]$. It follows that $g$ is a cw-expansive homeomorphism and that $\mathbb{T}^2/\sim$ is homeomorphic to $\mathbb{T}^2$ (see Section 3 in \cite{AAV}) so we can assume that $g$ is a homeomorphism of $\mathbb{T}^2$. We will prove that $g$ does not satisfy the cw-local-product-structure.

\begin{proposition} \label{notcw}
The homeomorphism $g\colon\mathbb{T}^2\to\mathbb{T}^2$ is not cw-hyperbolic.
\end{proposition}

\begin{proof}
The main thing we need to observe is that for each $n\in\N$ both $C^u_\eps(u_n)$ and $C^s_\eps(u_n)$ are contained in the region between the hyperboles
$$xy=\frac{1}{2^{2n-1}}\,\,\,\,\,\, \text{and} \,\,\,\,\,\, xy=\frac{1}{2^{2n+1}}$$ 
since $f$ restricted to these hyperboles equals $T$ and, hence, points of these hyperboles will be distant from $u_n$ when iterated by $f$ and $f^{-1}$ (note that these hyperboles are invariant by $T$).
Thus, 
$$C^u_\eps(u_n)\cap C^s_\eps(u_m)= \emptyset \,\,\,\,\,\, \text{for every} \,\,\,\,\,\, n,m\in\N \,\,\,\,\,\, \text{with} \,\,\,\,\,\, n\neq m$$ and $f$ does not satisfy the cw-local-product-structure. To see that the same happens to $g$, we need to check that if $n\neq m$ then
$$C^u_\eps(x)\cap C^s_\eps(y)= \emptyset \,\,\,\,\,\, \text{for every} \,\,\,\,\,\, x\sim u_n \,\,\,\,\,\, \text{and} \,\,\,\,\,\, y\sim u_m$$
because this would imply that
$$C^u_\eps([u_n])\cap C^s_\eps([u_m])= \emptyset.$$
To prove this, we note that if $x\sim u_n$, then $x\in C^u_\eps(u_n)\cap C^s_\eps(u_n)$ and this implies that $x$ belongs to the region between the hyperboles, as observed above. Then the same argument assures that both $C^u_\eps(x)$ and $C^s_\eps(x)$ are contained in the region between the hiperboles. Since this happens for every $n\in\N$, this is enough to prove that
$$C^u_\eps([u_n])\cap C^s_\eps([u_m])= \emptyset\,\,\,\,\,\, \text{for every} \,\,\,\,\,\, n,m\in\N \,\,\,\,\,\, \text{with} \,\,\,\,\,\, n\neq m$$ and, hence, $g$ does not satisfy the cw-local-product structure.
\end{proof}

\section{Cw-hyperbolicity on $\mathbb{S}^2$}

We now start to discuss how to classify cw-hyperbolic homeomorphisms of $\mathbb{S}^2$. The results of \cite{ACCV} and \cite{ACCV2} and also Theorem \ref{teoChHypLsh} and Corollary \ref{classes} above, seem to indicate that either cw-hyperbolic homeomorphisms are expansive, or they are similar to linear cw-Anosov homeomorphisms of $\mathbb{S}^2$. It is proved in \cite{Hiraide} that topologically hyperbolic homeomorphisms of $\mathbb{T}^2$ are topologically conjugate to a linear Anosov diffeomorphism. It seems natural to ask if a similar result is true in the cw-hyperbolic case.

\begin{question}
Is every cw-hyperbolic homeomorphism of $\mathbb{S}^2$ topologically conjugate to a linear cw-Anosov diffeomorphism?
\end{question}

We answer this question affirmatively for homeomorphisms of $\mathbb{S}^2$ that are induced by topologically hyperbolic homeomorphisms of $\mathbb{T}^2$.

\begin{theorem}\label{sphere}
If $f\colon \mathbb{T}^2\to\mathbb{T}^2$ is a topologically hyperbolic homeomorphism that induces a homeomorphism $g\colon\mathbb{S}^2\to\mathbb{S}^2$, then $g$ is cw-hyperbolic and is topologically conjugate to a linear cw-Anosov diffeomorphism.
\end{theorem}

\begin{proof}
In this case, there is a conjugacy between $f$ and a linear Anosov diffeomorphism $f_A$ given by \cite{Hiraide}, we only need to prove that the conjugacy homeomorphism induces a homeomorphism of $\mathbb{S}^2$ that conjugates $g$ and $g_A$. To prove this we need to go back to techniques of Franks \cite{F1} that create a semiconjugacy between any homeomorphism of $\mathbb{T}^2$ to a linear Anosov diffeomorphism provided it is homotopic to it. We will adapt some steps of his proof to ensure that this semiconjugacy map induces not only a map of $\mathbb{T}^2$ but also a map of $\mathbb{S}^2$. It is proved in \cite{Hiraide} that the linear part $A$ of $f$ is hyperbolic, so that $f$ is homotopic to a linear Anosov diffeomorphism $f_A$.
We will also denote by $f$ any lift to $\R^2$ of the map $f$ of $\mathbb{T}^2$ and we hope this causes no confusion. Note that $f$ induces a homeomorphism of $\mathbb{S}^2$ if, and only if, $f(-x)=-f(x)$ for every $x\in\R^2$.
Let $C^0(\R^2)$ denote the space of all continuous maps of $\R^2$ and define $F:C^0(\R^2)\to C^0(\R^2)$ by
$$F(h)=A^{-1} \circ h \circ f$$ where $A$ denotes the linear part of $f$. Note that $F(h)=h$ if, and only if, $A\circ h=h\circ f$, so that fixed points of $F$ are in a bijective relation with semiconjugacies between $A$ and $f$. In \cite{F1} Franks obtains a fixed point of $F$ of the form $h+id$ where $h\in Q$ and
$$Q=\{h\in C^0(\R^2); \ h(x+m)=h(x) \text{ for every } x \in \R^2 \text{ and } m \in \Z^2\}.$$
assuring that $h+id$ induces a map of $\mathbb{T}^2$. To ensure it induces a map on $\mathbb{S}^2$ we consider the set
$$Q'=\{h\in Q; \ h(-x)=-h(x) \,\,\, \text{for every} \,\,\, x\in\R^2 \}$$
and we prove that $F$ has a fixed point of the form $h+id$ with $h\in Q'$. Note that $Q$ is a vector space and that $Q'$ is a subspace of $Q$ that is invariant by $F$. Indeed, if $h_1$, $h_2 \in Q'$ and $t \in \R$, then
\begin{align*}
(h_1+th_2)(-x) & = h_1(-x)+th_2(-x)\\ & =-h_1(x)-th_2(x) \\ & = -(h_1+th_2)(x)
\end{align*}
for every $x\in\R^2$, so that $h_1+th_2\in Q'$.
Also if $h\in Q'$, then
\begin{align*}
F(h)(-x) & = A^{-1} \circ h \circ f(-x) \\ & = A^{-1} \circ h(-f(x)) \\ & = -A^{-1}(h(f(x)))\\ & = -F(h)(x)
\end{align*}
for every $x\in\R^2$, so that $F(h)\in Q'$. Noted this, we can consider the restriction $F: Q' \to Q'$. Now we note that $F$ is a linear map in $Q'$, since
\begin{align*}
F(h_1+th_2) & = A^{-1} \circ (h_1+th_2) \circ f\\ & =A^{-1} \circ h_1 \circ f + t A^{-1} \circ h_2 \circ f \\ & = F(h_1) + tF(h_2).
\end{align*}
As is \cite{F1}, we denote by $I$ the identity map on $C^0(\R^2)$ and prove that $$(F-I)(-id) \in Q'.$$ Indeed, we have
\begin{align*}
(F-I)(-id)(-x) & = F(-id)(-x) - I(-id)(-x) \\ & = A^{-1} \circ (-id) \circ f(-x) - (-id)(-x)\\ & = A^{-1}(f(x))-x \\ &  = F(id)(x) - I(id)(x)\\ & = -(F-I)(-id)(x).
\end{align*}
It is proved in \cite{F1} that $(F-I)$ is an isomorphism of $Q$ and we will use a simiar argument to assure that it is an isomorphism of $Q'$. Indeed, since $A$ is hyperbolic, we can write $\mathbb{R}^2=E^s\oplus E^u$ and choose $0<\lambda<1$ satisfying $$|A^n(v)|\leq\lambda^n|v|, \,\,\,\,\,\, \text{for every} \,\,\,\,\,\, v\in E^s \,\,\,\,\,\, \text{and}$$
$$|A^{-n}(v)|\leq\lambda^n|v|, \,\,\,\,\,\, \text{for every} \,\,\,\,\,\, v\in E^u$$ and also write $Q'=Q'^s\oplus Q'^u$ where $$Q'^s=\{h\in Q'; h(\mathbb{R}^2)\subset E^s\} \,\,\,\,\,\, \text{and}$$ $$Q'^u=\{h\in Q'; h(\mathbb{R}^2)\subset E^u\}.$$ Note that $F(Q'^s)=Q'^s$ because $h\in Q'^s$ implies $h\circ f(\mathbb{R}^2)\subset E^s$ and hence
$$F(h)(\mathbb{R}^2)=A^{-1}\circ h\circ f(\mathbb{R}^2)\subset E^s.$$ A similar argument assures that $F(Q'^u)=Q'^u$. Considering the $C^0$ norm $\|.\|$ in $Q'$ (here we recall that the definition of $Q$ assures that the $C^0$ norm is well defined) it is easily checked that
$$\|F^n(h)\|\leq \lambda^n\|h\|, \,\,\,\,\,\, \text{for every} \,\,\,\,\,\, h\in Q'^u \,\,\,\,\,\, \text{and}$$
$$\|F^{-n}(h)\|\leq \lambda^n\|h\|, \,\,\,\,\,\, \text{for every} \,\,\,\,\,\, h\in Q'^s,$$ where $F^{-1}(h)=A\circ h\circ f^{-1}$. Thus, we can prove that $$[(F-I)_{|Q'^u}]^{-1}=-\sum_{n=0}^{+\infty}(F_{|Q'^u})^n \,\,\,\,\,\, \text{and}$$
$$[(F-I)_{|Q'^s}]^{-1}=\sum_{n=1}^{+\infty}(F_{|Q'^s})^{-n}.$$ Note that the previous inequalities assure that the right side of these equalities converge. We have $$(F-I)_{|Q'^u}\circ\left(-\sum_{n=0}^{+\infty}(F_{|Q'^u})^n\right)=\sum_{n=0}^{+\infty}(F_{|Q'^u})^n-\sum_{n=1}^{+\infty}(F_{|Q'^u})^n=I_{|Q'^u}$$ and also that
$$(F-I)_{|Q'^s}\circ \sum_{n=1}^{+\infty}(F_{|Q'^s})^{-n}=\sum_{n=0}^{+\infty}(F_{|Q'^s})^{-n}-\sum_{n=1}^{+\infty}(F_{|Q'^s})^{-n}=I_{|Q'^s}.$$
Since $Q'=Q'^s\oplus Q'^u$ it follows that $F-I$ is an isomorphism of $Q'$. Thus, as $(F-I)(-id) \in Q'$, there exists $h \in Q'$ such that $$(F-I)(h) = (F-I)(-id).$$ Thus, $h+id$ is a fixed point of $F$ because $$(F-I)(h+id) = (F-I)(h)-(F-I)(-id) = 0$$ which, in turn, implies $$F(h+id)=h+id.$$ This proves that $h+id$ induces a semiconjugacy on $\mathbb{S}^2$ between $g$ and $g_A$. The injectivity of this map follows from the proof of the topologically hyperbolic case proved in \cite{Hiraide}.
\end{proof}

We exhibit a family of non-linear cw-hyperbolic homeomorphisms of $\mathbb{S}^2$ obtained as quocients of topologically hyperbolic homeomorphisms of $\mathbb{T}^2$, illustrating the previous theorem. For each $t\in[0,1]$ define $f_t\colon\R^2\to\R^2$ as
$$f_t(x,y)=\left(2x+y-\frac{t}{2\pi}\sin(2\pi x), x+y-\frac{t}{2\pi}\sin(2\pi x)\right).$$
Note that $f_0$ induces in $\mathbb{T}^2$ the usual linear Anosov diffeomorphism and in $\mathbb{S}^2$ the usual linear cw-Anosov diffeomorphism. Also, $f_t$ induces homeomorphisms in $\mathbb{T}^2$ and $\mathbb{S}^2$ for every $t\in[0,1]$ since 
$$f_t(x+m_1,y+m_2)=f_t(x,y)+(2m_1+m_2,m_1+m_2) $$  $$\text{and} \,\,\,\,\,\,  f_t(-x,-y)=-f_t(x,y)$$
for every $(x,y)\in\R^2$ and every $(m_1,m_2)\in\Z^2$.
The familiy in $\mathbb{T}^2$ induced by $(f_t)_{t\in[0,1]}$ is formed by topologically hyperbolic homeomorphisms as proved in \cite{Lewowicz} and Theorem \ref{sphere} assures that the family induced in $\mathbb{S}^2$ is formed by cw-hyperbolic homeomorphisms topologically conjugate to the linear cw-Anosov diffeomorphism induced by $f_0$. One interesting fact is that $(0,0)$ is a fixed point for the whole family $(f_t)_{t\in[0,1]}$ that is not hyperbolic for $f_1$
since
$$Df_1(0,0)=\begin{pmatrix}
		1 & 1 \\
		0 & 1
		\end{pmatrix}.$$
and there is not a $Df_1(0,0)$-invariant splitting of $\R^2$, but even so $f_1$ induces a cw-hyperbolic homeomorphism of $\mathbb{S}^2$ that is not topologically hyperbolic. A consequence of the next proposition is that these examples given by Theorem \ref{sphere} are cw2-hyperbolic and not cw1-hyperbolic since the linear cw-Anosov diffeomorphisms are cw2-hyperbolic but not cw1-hyperbolic.

\begin{proposition}
Let $f$ be a homeomorphism of a compact metric space $X$ and $g$ be a homeomorphism of a compact metric space $Y$. If $f$ and $g$ are topologically conjugate, then 
\begin{enumerate}
\item $f$ is cw-hyperbolic if, and only if, $g$ is cw-hyperbolic,
\item $f$ is cwN-expansive if, and only if, $g$ is cwN-expansive.
\end{enumerate}
\end{proposition}

\begin{proof}
Let $d_X$ and $d_Y$ denote the metrics in $X$ and $Y$, respectively, and $h\colon X\to Y$ be the topological conjugacy between $f$ and $g$. Suppose that $f$ is cw-hyperbolic with $c>0$ a cw-expansivity constant. To prove that $g$ is cw-hyperbolic we need to prove that $g$ is cw-expansive and it satisfies the cw-local-product-structure. Choose $\alpha\in(0,c)$, given by uniform continuity of $h^{-1}$, such that
$$d_Y(x,y)<\alpha \,\,\,\,\,\, \text{implies} \,\,\,\,\,\, d_X(h^{-1}(x),h^{-1}(y))<c.$$
If $\diam_Y(g^k(C))\leq\alpha$ for every $k\in\Z$, then $\diam_X(f^k(h^{-1}(C)))\leq c$ for every $k\in\Z$, and, hence, $h^{-1}(C)$ is totally disconnected since $f$ is cw-expansive. It follows that $C$ is totally disconnected, since $h$ is a homeomorphism, and that $g$ is cw-expansive and $\alpha$ is a cw-expansivity constant of $g$. Let $\eps>0$ be given and choose $\eps'\in(0,\eps)$, given by uniform continuity of $h$, such that
$$d_X(x,y)<\eps' \,\,\,\,\,\, \text{implies} \,\,\,\,\,\, d_X(h(x),h(y))<\eps.$$
Let $\delta'\in(0,\eps')$ be given by the cw-local-product-structure of $f$, such that
$$d_X(x,y)<\delta' \,\,\,\,\,\, \text{implies} \,\,\,\,\,\, C^u_{\eps'}(x)\cap C^s_{\eps'}(y)\neq\emptyset.$$ Let $\delta>0$, given by uniform continuity of $h^{-1}$, be such that
$$d_Y(x,y)<\delta \,\,\,\,\,\, \text{implies} \,\,\,\,\,\, d_X(h^{-1}(x),h^{-1}(y))<\delta'.$$
Thus, if $d_Y(x,y)<\delta$, then $d_X(h^{-1}(x),h^{-1}(y))<\delta'$ and, hence, there exists $$z\in C^u_{\eps'}(h^{-1}(x))\cap C^s_{\eps'}(h^{-1}(y)).$$ The choice of $\eps'$ assures that $h(z)\in C^u_{\eps}(x)\cap C^s_{\eps}(y)$ and this proves (1). To prove (2) assume that $f$ is cwN-expansive and $r>0$ satisfies $$\#(C^s_r(x)\cap C^u_r(x))\leq N \,\,\,\,\,\, \text{for every} \,\,\,\,\,\, x\in X.$$ Choose $\beta\in(0,r)$, given by uniform continuity of $h^{-1}$, such that $$d_Y(x,y)<\beta \,\,\,\,\,\, \text{implies} \,\,\,\,\,\, d_X(h^{-1}(x),h^{-1}(y))<r.$$ Then
$$h^{-1}(C^s_{\beta}(y)\cap C^u_{\beta}(y))\subset C^s_r(h^{-1}(y))\cap C^u_r(h^{-1}(y)) \,\,\,\,\,\, \text{for every} \,\,\,\,\,\, y\in Y$$
and since $$\#(C^s_r(h^{-1}(y))\cap C^u_r(h^{-1}(y)))\leq N$$ and $h$ is a homeomorphism, it follows that 
$$\#(C^s_{\beta}(y)\cap C^u_{\beta}(y))\leq N \,\,\,\,\,\, \text{for every} \,\,\,\,\,\, y\in Y.$$ This proves that $g$ is cwN-expansive and (2) is proved.
\end{proof}

\begin{remark}
The following is equivalent to the item (2) above: $f$ is not cwN-expansive if, and only if, $g$ is not cwN-expansive. Thus, $f$ is cwN-hyperbolic but not cw(N-1)-expansive if, and only if, $g$ is cwN-hyperbolic but not cw(N-1)-expansive.
\end{remark}

\begin{question}
Does there exist a cwN-hyperbolic homeomorphism of a closed surface that is not cw2-expansive?
\end{question}

\section*{Acknowledgements}
The second author was supported by CAPES and the Alexander von Humboldt Foundation under the project number 88881.162174/2017-1 and also by CNPq grant number 405916/2018-3.

\end{document}